\documentclass[10pt,reqno,oneside]{amsart}

\usepackage{lscape}
\usepackage{amsfonts}
\usepackage{pdfsync}
\usepackage{hyperref}
\usepackage{color}
\usepackage{amsmath}
\usepackage[colorinlistoftodos]{todonotes} 

\newcommand*\Eval[3]{\left.#1\right\rvert_{#2}^{#3}}

\newtheorem*{theorem-non}{Theorem}
\newtheorem*{lemma-non}{Lemma}
\newtheorem{theorem}{Theorem}[section]
\newtheorem{lemma}[theorem]{Lemma}
\newtheorem{corollary}[theorem]{Corollary}

\newtheorem{example}[theorem]{Examples}

\theoremstyle{definition}
\newtheorem{definition}[theorem]{Definition}
\newtheorem{remark}[theorem]{Remark}
\newtheorem{theorem*}[theorem]{Theorem}

\def\O{\Omega}

\def\C{\mathcal{C}}
\def\d{{\rm d}}

\def\div{{\rm div}\,}
\def\supp{{\rm supp}}
\def\dx{{\,\rm d}x}

\def\u{{\bf u}}
\def\v{{\bf v}}

\newcommand{\R}{{\mathbb R}}
\newcommand{\N}{{\mathbb N}}

\DeclareMathOperator*{\esssup}{ess\,sup}
\DeclareMathOperator*{\essinf}{ess\,inf}

\begin{document}
\section*{}
\title[An application of the Hardy inequality]{An application of the weighted discrete Hardy inequality}

\author[Bui]{Giao Bui}
\author[L\'opez-Garc\'ia]{Fernando L\'opez-Garc\'ia}
\author[Tran]{Van Tran}
\address{Department of Mathematics and Statistics\\ California State Polytechnic University Pomona, 
3801 West Temple Avenue, Pomona, CA (91768), US} 
\email{gqbui@cpp.edu}
\email{fal@cpp.edu}
\email{vttran@cpp.edu \& van.tran@umconnect.umt.edu}

\thanks{Partially Supported by NSF-DMS 1247679 grant PUMP: Preparing Undergraduates  through Mentoring towards PhDs}

\keywords{Discrete Hardy inequality, Divergence operator, Decomposition of functions, Cusps, Weights, Stokes equations}

\subjclass[2010]{Primary: 26D15; Secondary: 46E35,76D07}

\begin{abstract}
In a note published in 1925, G. H. Hardy stated the inequality 
\begin{equation*}
\sum_{n=1}^\infty \left(\frac{1}{n}\sum_{k=1}^n a_k \right)^p \leq \left(\frac{p}{p-1}\right)^p \sum_{n=1}^\infty a_n^p,
\end{equation*}
for any non-negative sequence $\{a_n\}_{n \geq 1}$, and $p>1$. This inequality is known in the literature as the classical discrete Hardy inequality. It has been widely studied and several applications and new versions have been shown. 

In this work, we use a characterization of a weighted version of this inequality to exhibit a sufficient condition for the existence of solutions of the differential equation $\div \u=f$ in weighted Sobolev spaces over a certain plane irregular domain. The solvability of this equation is fundamental for the analysis of the Stokes equations. 

The proof follows from a local-to-global argument based on a certain decomposition of functions which is also of interest for its applications to other inequalities or related results in Sobolev spaces, such as the Korn inequality.
\end{abstract}

\maketitle

\section{Introduction}\label{Intro}

Given $p>1$, {\it the discrete Hardy inequality} states 

\begin{equation}\label{Hardy}
\sum_{n=1}^\infty \left(\frac{1}{n}\sum_{k=1}^n a_k \right)^p \leq \left(\frac{p}{p-1}\right)^p \sum_{n=1}^\infty a_n^p,
\end{equation}
for any non-negative sequence $\{a_n\}_{n\geq 1}$, where the constant in the inequality $(p/(p-1))^p$ is optimal. This inequality has been widely studied and many generalizations have been shown. In this article, we use a weighted version known as {\it the weighted discrete Hardy inequality} which says: 
\begin{equation}\label{wHardy}
\left(\sum_{n=1}^\infty u_n\left(\sum_{k=1}^n a_k \right)^p \right)^{1/p} \leq C \left( \sum_{n=1}^\infty v_n a_n^p \right)^{1/p}.
\end{equation} 
The existence of a constant $C$, that makes inequality \eqref{wHardy} valid for any non-negative sequence $\{a_n\}_{n\geq 1}$, depends only on $p$ and the sequence weights $\{u_n\}_{n\geq 1}$ and $\{v_n\}_{n\geq 1}$. There are several characterizations of the sequence weights in the previous inequality such as the one published in \cite{AH} that states that the constant $C$ in \eqref{wHardy} exists if and only if 
\begin{equation*}
A=\sup_{k\geq 1} \left(\sum_{i=k}^\infty u_i\right)^{1/p} \left(\sum_{i=1}^k v_i^{1-q} \right)^{1/q} <\infty,
\end{equation*}
where $q=p/(p-1)$. See also \cite{B,O} for more information about this type of inequalities. The existence of a characterization for the sequence weights in \eqref{wHardy} is key to prove our main result on the solvability of the divergence equation in weighted Sobolev spaces. We deal with the existence of weighted Sobolev solutions of the equation $\div\u=f$ for weights $\nu_1(x), \nu_2(x):\Omega\to \R_{>0}$, where $\Omega$  is the planar domain
\begin{equation}\label{domain}
\Omega:=\{(x_1,x_2)\in\R^2\, : \, 0<x_1<1 \text{ and } 0<x_2<x_1^\gamma\},
\end{equation} 
\begin{figure}[h]
\begin{tikzpicture}[xscale=4,yscale=3]
\draw[thick, fill=gray!50!white] plot[smooth,samples=100,domain=0:1] (\x,{(\x)^2}) node[above right]{\large $x_2=x_1^\gamma$} -- 
    plot[smooth,samples=100,domain=1:0] (\x,{0});
\draw[<->] (-.2,0)--(1.5,0) node[right]{$x_1$};
\draw[<->] (0,-0.2)--(0,1.2) node[above]{$x_2$};
\foreach \x in {1}
    \draw (\x,2pt)--(\x,-2pt) node[below] {$\x$};
\foreach \y/\ytext in {1}
    \draw (2pt,\y)--(-2pt,\y) node[left] {$\y$};    
\node at (.75,.25) {\Large $\Omega$};
\end{tikzpicture}
\label{Omega}
\end{figure}
for $\gamma\geq 1$. Specifically, we are looking for sufficient conditions on the weights $\nu_1(x)$ and $\nu_2(x)$ such that, for any $f\in L^2(\Omega,\nu_2(x))$ with vanishing mean value, there exists a solution $\u$ of $\div\u=f$ in the Sobolev spaces $H^1_0(\Omega,\nu_1(x))^2:=\overline{C_0^\infty(\Omega)^2}$ with the following estimate 
\begin{equation}\label{westimate}
\int_\Omega |D\u(x)|^2 \nu_1(x) \dx \leq C^2\int_{\Omega}|f(x)|^2\nu_2(x) \dx,
\end{equation}
where $D\u(x)$ denotes the differential matrix of $\u$. The weights considered here satisfy that $\nu_1(x)=x_1^{2(\gamma-1)}\nu_2(x)$ and $\nu_1,\nu_2$ depend only on the first component of $x$ (i.e. $\nu_1(x)=\nu_1(x_1)$ and  $\nu_2(x)=\nu_2(x_1)$).  Notice that if $\gamma>1$, the domain $\Omega$ has a singularity (cusp) at the origin, while the domain is regular (convex) if $\gamma=1$. The factor $x_1^{2(\gamma-1)}$ in the definition of $\nu_1(x)$ is there to deal with the singularity at the origin and disappears when $\Omega$ is regular ($\gamma=1$), in which case we have the same weights in both sides of the estimate \eqref{westimate}. The exponent in the factor $x_1^{2(\gamma-1)}$ is optimal in the following sense: if $\nu_2(x)=1$ and $\nu_1(x)=x_1^{a}$, with $a<2(\gamma-1)$, the solvability of $\div\u=f$ with estimate \eqref{westimate} fails in general (we refer to \cite{ADL} for counterexamples). 

The solvability of the divergence equation is fundamental for the variational analysis of the Stokes equations and strongly depends on the geometry of the domain, which has been studied in Lipschitz domains, star-shaped domains with respect to a ball, John domains, H\"older-$\alpha$ domains, among others. We refer to \cite{AD} and references therein for an extensive description of the solvability of this equation on domains under several geometric conditions. The domain $\Omega$ of our interest and defined in \eqref{domain} was already considered in \cite{DL,L1}. The authors in \cite{DL} use the Piola transform of an explicit solution on a regular domain whose analysis required the use of the theory of singular integral operators and Muckenhoupt weights. In \cite{L1}, the author uses a technique similar to the one treated in this article, where the discrete weighted Hardy inequality \eqref{wHardy} is replaced by a Hardy-type operator on weighted $L^p(\Omega)$ spaces. The reason to work with \eqref{wHardy} instead of the Hardy-type operator defined in \cite{L1} relies on the simplicity of the discrete inequality and the characterization of the weights for which the inequality remains valid.

Now, in order to prove our main results, we decompose $\Omega$ into a collection of infinitely many regular (star-shaped with respect to a ball) subdomains $\{\Omega_i\}_{i\geq 0}$ where the weights can be assumed to be constant. In that case the solvability of the divergence equation has been proved. Then, we extend by zero the solutions in $\Omega_i$ to the whole domain and add them up to obtain a solution in $\Omega$. Inequality \eqref{wHardy} appears when we estimate the norm of the ``global solution" in terms of the estimation of the ``local solutions". The decomposition $\{\Omega_i\}_{i\geq 0}$ of $\Omega$ mentioned above is:
\begin{equation}\label{Partition}
\Omega_i:=\{(x_1,x_2)\in\Omega\, :\, 2^{-(i+2)}<x_1<2^{-i}\}.
\end{equation}

This is the main result of the paper.
\begin{theorem}\label{Main}
Let $\omega:\Omega\to\R$ be an admissible weight in the sense of Definition \ref{admissible weights}, for $p=2$, such that the following weighted Hardy inequality is valid for any non-negative sequence $\{d_n\}_{n\geq 1}$: 
\begin{equation*}
\sum_{j=1}^\infty u_i \left(\sum_{i=1}^j d_i \right)^2\leq C_H^2\sum_{j=1}^\infty u_j d_j^2,
\end{equation*}
where 
\[u_i:=|\Omega_i|\,\omega^2(2^{-i}).\] 
Then, there exists a constant $C$ such that for any $f$ in $L^2(\Omega,\omega^{-2}(x_1))$, with vanishing mean value, there exists a solution $\u:\Omega\subset\R^2\to \R^2$ of the equation $\div\u=f$ in $H^1_0(\Omega, x_1^{2(\gamma-1)}\omega^{-2}(x_1))^2$  such that 
\begin{equation*}
\int_\Omega |D\u(x)|^2 x_1^{2(\gamma-1)}\omega^{-2}(x_1) \dx \leq C^2\int_{\Omega}|f(x)|^2\omega^{-2}(x_1) \dx.
\end{equation*}

Moreover, 
\begin{equation*}
C^2\leq \gamma^2 2^{12+4\gamma} C_\omega^8 C_H^2.
\end{equation*}
\end{theorem}

\begin{remark} The strong connection between the solvability of the equation $\div\u=f$ and the validity of the Korn inequality in the second case is well-known (see \cite{HP,JK,AD}).  Thus, it is worth observing that in \cite{AO} the authors use the weighted discrete Hardy inequality \eqref{wHardy} to prove the validity of the Korn inequality on domains with a single singularity on the boundary by using a different local-to-global argument. 
\end{remark}

The following result considers the case where the weights are power functions.

\begin{corollary}[Power weights]\label{Corollary 1}
Let $\O\subset\R^2$ be the domain defined in \eqref{domain} and $\beta>\frac{-\gamma-1}{2}$.  Then, there exists a positive constant $C$ such that for any $f\in L^2(\O,\omega(x_1)^{-2})$, with $\int_\O f=0$, there exists a solution $\u\in H^1_0\left(\O,x_1^{2(\gamma-1)}\omega(x_1)^{-2})\right)^2$ of ${\rm div}\,\u=f$ that satisfies
\begin{equation}\label{Ineq Cor 1}
\int_\Omega |D\u(x)|^2 x_1^{2(\gamma-1)}\omega(x_1)^{-2} \dx \leq C^2\int_{\Omega}|f(x)|^2 \omega(x_1)^{-2} \dx,
\end{equation}
where $\omega(x_1):=x_1^{\beta} $.
Moreover, if $\beta\leq 0$, the constant $C$ in \eqref{Ineq Cor 1} satisfies the following estimate:
\begin{equation*}
C\leq \dfrac{M}{1-2^{-2\left(\beta +\frac{\gamma+1}{2}\right)}},
\end{equation*}
where the constant $M$ is independent of $\beta$.
\end{corollary}

Notice that the distance from $(x_1,x_2)$ in $\Omega$ to the origin is comparable to $x_1$, thus the weights here can be understood as powers of the distance to the origin or the cusp if $\gamma>1$. Indeed, 
\[x_1\leq \sqrt{x_1^2+x_2^2}\leq \sqrt{2}x_1.\]
for all $(x_1,x_2)\in\Omega$. 

The existence of a solution of the divergence equation in this planar domain $\Omega$ with the estimate \eqref{Ineq Cor 1} was first obtained in \cite[Theorem 4.1]{DL} for $\beta$ in $\left(\frac{-\gamma-1}{2},\frac{3\gamma-1}{2}\right)$, and later in \cite[Theorem 5.1]{L1} for $\beta\geq 0$. In this case, we recover both results as a corollary of our main theorem. In addition, an estimate of the constant that bounds its blow-up as $\beta$ tends to $\frac{-\gamma-1}{2}$ is exhibited. 
Finally, notice that if $\beta \leq \frac{-\gamma-1}{2}$ then $L^2(\O,x_1^{-2\beta})\not\subset L^1(\Omega)$ and the vanishing mean value condition in the divergence problem is not well-defined. Hence, the condition $\beta>\frac{-\gamma-1}{2}$ is optimal for the current setting. For an example of a non-integrable function in $L^2(\O,x_1^{-2\beta})$, when $\beta\leq \frac{-\gamma-1}{2}$, one can consider $f(x)=(1-\ln(x_1))^{-1} x_1^{-\gamma-1}$.

The following result considers the case where the weights are powers of a logarithmic function. 

\begin{corollary}[Powers of logarithmic weights]\label{Corollary 2}
Let $\O\subset\R^2$ be the domain defined in \eqref{domain} and $\alpha\in\R$.  Then, there exists a positive constant $C$ such that for any $f\in L^2(\O,\omega(x_1)^{-2})$, with $\int_\O f=0$, there exists a solution $\u\in H^1_0\left(\O,x_1^{2(\gamma-1)}\omega(x_1)^{-2})\right)^2$ of ${\rm div}\,\u=f$ that satisfies
\begin{equation*}
\int_\Omega |D\u(x)|^2 x_1^{2(\gamma-1)}\omega(x_1)^{-2} \dx \leq C^2\int_{\Omega}|f(x)|^2 \omega(x_1)^{-2} \dx,
\end{equation*}
where $\omega(x_1):=(1-\ln(x_1))^{\alpha} $.
\end{corollary}

The article is organized as follows: In Chapter \ref{Decomposition}, we show that the weighted discrete Hardy inequality, with some appropriate weights, implies the validity of a certain decomposition of functions in which our local-to-global argument is based. The main result in this chapter might be of interest for applications to other inequalities and related results in Sobolev spaces. In this chapter, we consider general $1<p,q<\infty$, with $\frac{1}{p}+\frac{1}{q}=1$. Then, we use the estimate of the constant in the divergence equation provided by Costabel and Dauge \cite{CD} for $p=q=2$ to prove Theorem \ref{Main}. In Chapter \ref{Selena Van Intro}, we prove the validity of the corollaries stated in the introduction that claim the solvability of the divergence equation in weighted spaces for power weights and powers of logarithmic weights.

The novelty of this work lies in the use of the well-studied weighted discrete Hardy inequality to get new sufficient conditions on the weights that imply the solvability of the divergence equation, recovering the existing results in \cite{DL,L1} when the weights are powers of the distance to the cusp/origin. The second corollary using powers of logarithmic weights is also new. 
 
\section{A decomposition of functions and applications}
\label{Decomposition}
\setcounter{equation}{0}

We name a {\it weight} $\nu:\Omega\to\R$ a positive and Lebesgue-measurable function, and a {\it sequence weight} $\{\nu_i\}_{i\geq 1}$ a sequence of positive real numbers. We will denote by $x=(x_1,x_2)$ a general point in $\R^2$.

\begin{definition}\label{admissible weights} A weight $\omega:\Omega\to \R$ is called {\it admissible} if $\omega^p\in L^1(\Omega)$ and there exists a uniform constant $C_\omega$ such that 
\begin{equation}\label{admissible}
\esssup_{x\in \Omega_i} \omega(x)\leq C_\omega \essinf_{x\in \Omega_i} \omega(x),
\end{equation}
for all $i\geq 0$. Notice that admissible weights are subordinate to a partition $\{\Omega_i\}_{i\geq 0}$ of $\Omega$ introduced in \eqref{Partition}, and $1<p<\infty$. 
\end{definition}

\begin{example} The function $\omega(x):=x_1^{\beta}$, where $\beta>\frac{-\gamma-1}{p}$, is an admissible weight with $C_\omega=2^{2|\beta_0|}$, where $\beta_0:=\frac{-\gamma-1}{p}$.
\end{example}

\begin{definition}\label{decomp} Given $g:\Omega\to\R$ integrable function with vanishing mean value, i.e. $\int g=0$, we refer by {\it a $\C$-orthogonal decomposition of $g$} subordinate to $\{\Omega_i\}_{i\geq 0}$ to a collection of integrable functions $\{g_i\}_{i\geq 0}$ with the following properties: 
\begin{enumerate}
\item $g(x)=\sum_{i\geq 0} g_i(x).$
\item $\supp (g_i)\subset \Omega_i,$ for all $i\geq 0$.
\item $\int_{\Omega_i} g_i =0$, for all $i\geq 0.$
\end{enumerate}
\end{definition}

The letter $\C$ in the previous definition refers to the space of constant functions. Notice that having vanishing mean value could also be  understood as being orthogonal to the functions in $\C$. Other applications of this type of decomposition of functions require to have orthogonality to other spaces (see \cite{L2,L3}). We also refer the reader to \cite{HL} for applications to a fractional Poincar\'e type inequality.

We show the existence of a $\C$-orthogonal decomposition by using a constructive argument introduced in \cite{L1}. \label{example} Let us describe the idea of this argument assuming that $\Omega$ is the union of the first three subdomains in partition defined in \eqref{Partition}. Thus, let $f\in L^1(\O)$ be a function with vanishing mean value. Then, using a partition of the unity $\{\phi_i\}_{0\leq i\leq 2}$ subordinate to $\{\O_i\}_{0\leq i\leq 2}$ we can write $g$ as: 
\[g=f_0+f_1+f_2=g\phi_0+g\phi_1+g\phi_2.\]
However, this partition might not be orthogonal to $\C$. In order to get this property we make the following arrangements:
\begin{align*}
g=f_0+\left(f_1+\dfrac{\chi_{B_2}}{|B_2|}\int_{\O_2}f_2\right) + \underbrace{\left(f_2-\dfrac{\chi_{B_2}}{|B_2|}\int_{\O_2}f_2\right)}_{f_2-h_2},
\end{align*}
where $B_2:=\O_2\cap \O_1$. Note that the function $f_2-h_2$ has its support in $\O_2$ and $\int f_2-h_2=0$. Finally, we repeat the process with the first two functions. Thus, if $B_1:=\O_1\cap \O_0$ we have that 
\begin{multline}\label{Bogovskii}
f=\overbrace{\left(f_0+\dfrac{\chi_{B_1}}{|B_1|}\int_{\O_1\cup\O_2}f_1+f_2\right)}^{f_0-h_0}\\ 
+\underbrace{\left(f_1+\dfrac{\chi_{B_2}}{|B_2|}\int_{\O_2}f_2-\dfrac{\chi_{B_1}}{|B_1|}\int_{\O_1\cup\O_2}f_1+f_2\right)}_{f_1-h_1} + \underbrace{\left(f_2-\dfrac{\chi_{B_2}}{|B_2|}\int_{\O_2}f_2\right)}_{f_2-h_2},
\end{multline}
obtaining the claimed decomposition. Observe that we have used the vanishing mean value of $f$ only to prove that $f_0-h_0$ integrates zero.

Now, let us introduce the following weighted discrete Hardy-type inequalities: 
\begin{equation}\label{dHardy 2}
\sum_{j=1}^\infty |\Omega_j| \omega_j^{p}\left(\sum_{i=1}^j d_i \right)^p\leq C_H^p\sum_{j=1}^\infty |\Omega_j| \omega_j^{p} d_j^p,
\end{equation}
and
\begin{equation}\label{dHardy 1}
\sum_{i=1}^\infty |\Omega_i|^{1-q} \omega_i^{-q}\left(\sum_{j=i}^\infty b_j \right)^q\leq C_H^q\sum_{i=1}^\infty |\Omega_i|^{1-q} \omega_i^{-q} b_i^q.
\end{equation}
The first one is inequality \eqref{wHardy} to the $p$ power where the sequence weight $u_n=v_n=|\Omega_n|\omega_n^p$, and the second one is its dual version. The following lemma follows from this duality.

\begin{lemma}\label{Equiv} Given a sequence weight $\{\omega_i\}_{i\geq 1}$, inequality (\ref{dHardy 1}) is valid for any non-negative sequence $\{b_i\}_{i\geq 1}$ if and only if inequality (\ref{dHardy 2}) is valid for any non-negative sequence $\{d_j\}_{j\geq 1}$, with the same constant $C_H$.
\end{lemma}
\begin{proof} By using the duality between $l^p$ and $l^q$, and defining $\tilde{d}_j:=|\Omega_j|^{1/p}\omega_j d_j$ and $\tilde{b}_i:=|\Omega_i|^{-1/p}\omega_i^{-1} b_i$, it follows that inequality \eqref{dHardy 2} and \eqref{dHardy 1} can be written as 
\begin{equation}\label{dHardy 2 bis}
\sup_{\|\tilde{d}\|_{l^p}=1}\sup_{\|\tilde{b}\|_{l^q}=1}\sum_{j=1}^\infty \tilde{b_j} |\Omega_j|^{1/p} \omega_j  \sum_{i=1}^j |\Omega_i|^{-1/p} \omega_i^{-1} \tilde{d}_i \leq C_H
\end{equation}
and
\begin{equation}\label{dHardy 1 bis}
\sup_{\|\tilde{b}\|_{l^q}=1}\sup_{\|\tilde{d}\|_{l^p}=1}\sum_{i=1}^\infty \tilde{d_i} |\Omega_i|^{-1/p} \omega_i^{-1}\sum_{j=i}^\infty |\Omega_j|^{1/p} \omega_j \tilde{b}_j \leq C_H
\end{equation}

Finally, one can obtain \eqref{dHardy 1 bis} from \eqref{dHardy 2 bis}, and viceversa, by changing the order of the summations.
\end{proof}

\begin{theorem}\label{Decomp Thm}  Let $\omega:\O\to\R$ be an admissible weight that satisfies \eqref{dHardy 2} for the sequence weight $\omega_i:=\omega(2^{-i})$. Then, given $g\in L^1(\Omega)$, with $\int_\O g=0$, there exists $\{g_t\}_{t\in\Gamma}$, a  $\C$-decomposition of $g$ subordinate to $\{\Omega_i\}_{i\geq 0}$ (see Definition \ref{decomp}), such that 
\begin{equation}\label{Decomp estim}
\sum_{i=0}^\infty \int_{\Omega_i}|g_i(x)|^q \omega^{-q}(x) \dx \leq C_d^q \int_{\Omega} |g(x)|^q \omega^{-q}(x)\dx.
\end{equation}
Moreover, we have the following estimate for the optimal constant $C_d$:
\begin{equation}\label{Decomp constant}
C_d\leq 2^{2+1/q}C_\omega^{2} C_H.
\end{equation}
\end{theorem}

\begin{proof} 
The decomposition  treated here follows the example with three subdomains in page \pageref{example}. Indeed, let $\{\phi_i\}_{i\geq 0}$ be a partition of unity subordinate to the collection $\{\O_i\}_{i\geq 0}$. Namely, a collection of smooth functions such that $\sum_{i\geq 0} \phi_i=1$, $0\leq \phi_i\leq1$ and $supp(\phi_i)\subset\Omega_i$. Thus, $g$ can be cut-off into $g=\sum_{i\geq 0} f_i$ by taking $f_i=g\phi_i$. This decomposition satisfies (1) and (2) in Definition \ref{decomp} but not necessarily (3). Thus, we make the following modifications to $\{f_i\}_{i\geq 0}$ to obtain a collection of functions that also satisfies (3). Indeed, for any $i\geq 1$,
\begin{align}\label{P0-decomposition}
g_i(x):=f_i(x)+h_{i+1}(x)-h_i(x), 
\end{align}
where 
\begin{align}\label{h}
h_i(x)&:=\dfrac{\chi_{i}(x)}{|B_{i}|}\int_{W_{i}}\sum_{k\geq i}f_k, \nonumber \\
B_i&:=\Omega_i\cap \Omega_{i-1},\\
W_i&:=\bigcup_{k\geq i} \Omega_k. \nonumber
\end{align}

We denote by $\chi_i$ the characteristic function of $B_i$. Notice that the auxiliary function $h_i$ is not defined for $i=0$, thus $g_0$ follows in this other way 
\[g_0(x)=f_0(x)+h_1(x).\]

This decomposition was introduced in \cite{L1} in a more general way where the natural numbers in the subindex set is replaced by a set with a partial order given by a structure of tree (i.e. connected graph without cycles). We also use in this article inequality \eqref{dHardy 2} instead of another Hardy type inequality on trees introduced in \cite{L1}. Thus, it only remains to show estimate \eqref{Decomp estim}. Notice that $h_i$ and $h_{i+1}$ have disjoint supports thus 
\begin{equation*}
|h_{i+1}(x)-h_i(x)|^q=|h_{i+1}(x)|^q+|h_{i}(x)|^q.
\end{equation*}

Next, using that $|a+b|^q\leq 2^{q-1}(|a|^q+|b|^q)$ for all $a,b\in\R$, we have
\begin{align}\label{equation1}&\sum_{i=0}^\infty \int_{\Omega_i}|g_i(x)|^q \omega^{-q}(x_1)\dx \notag\\
\leq\, & 2^{q-1}\left(\sum_{i=0}^\infty \int_{\Omega_i}|f_i(x)|^q\omega^{-q}(x_1)\dx+2\sum_{i=1}^\infty \int_{\Omega_i}|h_i(x)|^q\omega^{-q}(x_1)\dx\right)\notag\\ 
\leq\, & 2^{q} \left(\int_{\Omega}|g(x)|^q\omega^{-q}(x_1)\dx+\sum_{i=1}^\infty \int_{\Omega_i}|h_i(x)|^q\omega^{-q}(x_1)\dx\right).
\end{align}

Let us work over the sum on the right hand side in the previous inequality by using the weighted discrete Hardy inequality. Notice that from the definition of the auxiliary functions in \eqref{h} and inequality \eqref{admissible} in Definition \ref{admissible weights} it follows that
\begin{align*}
|h_i(x)|\leq \dfrac{\chi_{i}(x)}{|B_{i}|}\sum_{k= i}^\infty \int_{\Omega_i}|g|
\end{align*}
and
\begin{align*}
\int_{\Omega_i}\dfrac{\chi_{i}(x)}{|B_{i}|^q}\omega^{-q}(x_1)\dx\leq C_\omega^q \omega_i^{-q} |B_i|^{1-q}.
\end{align*}

Therefore, since $|\Omega_i|<2|B_i|$ for any $i\geq 1$, the sum in inequality \eqref{equation1} is bounded by 
\begin{align*}
\sum_{i=1}^\infty \int_{\Omega_i}|h_i(x)|^q\omega^{-q}(x_1)\dx &\leq C_\omega^{q} \sum_{i=1}^\infty |B_i|^{1-q}\omega^{-q}_i  \left(\sum_{k=i}^\infty \int_{\Omega_k} |g|\right)^q\\ 
&\leq 2^{q-1}C_\omega^q \sum_{i=1}^\infty  |\Omega_i|^{1-q} \omega^{-q}_i\left(\sum_{k=i}^\infty \int_{\Omega_k} |g|\right)^q.
\end{align*}
Next, by using Lemma \ref{Equiv} with $b_i=\int_{\Omega_i} |g|$, $i\geq 1$, and H\"older inequality, we can conclude that 
\begin{align*}
\sum_{i=1}^\infty \int_{\Omega_i}|h_i(x)|^q\omega^{-q}(x_1)\dx &\leq 2^{q-1}C_\omega^q C_H^q \sum_{i=1}^\infty |\Omega_i|^{1-q} \omega^{-q}_i  \left(\int_{\Omega_i} |g|\right)^q\\ 
&\leq 2^{q-1}C_\omega^q C_H^q \sum_{i=1}^\infty \omega^{-q}_i  \left(\int_{\Omega_i} |g|^q\right)\\
&\leq 2^{q-1}C_\omega^{2q} C_H^q \sum_{i=1}^\infty  \int_{\Omega_i} |g(x)|^q\omega^{-q}(x_1)\dx\\
&\leq 2^qC_\omega^{2q} C_H^q \int_{\Omega} |g(x)|^q\omega^{-q}(x_1)\dx.
\end{align*}

Finally, from inequality \eqref{equation1} it follows \eqref{Decomp estim}. 
\end{proof}

In order to prove the solvability of the divergence equation on the subdomains $\Omega_i$ we use the following result proved by M. Costabel and M. Dauge in \cite{CD} for star-shaped domains. Let us recall the definition of this class of domains. A domain $U$ is {\it star-shaped with respect to a ball $B$} if and only if any segment with an end-point in $U$ and the other one in $B$ is contained in $U$. 

\begin{theorem}\label{Costabel Dauge Thm}
Let $U\subset\R^2$ be a domain contained in a ball of radius $R$, star-shaped with respect to a concentric ball of radius $r$. Then, for any $g\in L^2(U)$ with vanishing mean value there exists a solution $\u\in H^1_0(U)^2$ of the equation $\div\u=g$ satisfying the estimate
\begin{equation*}
\left(\int_U |D\u(x)|^2\dx\right)^{1/2}\leq \frac{2R}{r} \left(\int_U |g(x)|^2\dx\right)^{1/2}.
\end{equation*}
\end{theorem}

\begin{proof}[Proof of Theorem \ref{Main}]  Let $f$ be a function in $L^2(\Omega,\omega^{-2}(x_1))$ with vanishing mean value. Notice that, since $\omega$ is an admissible weight for $p=2$, $L^2(\Omega,\omega^{-2}(x_1))\subset L^1(\Omega)$ and the mean value of $f$ is well-defined. Then, from Theorem \ref{Decomp Thm}, there exists a $\C$-decomposition $\{f_i\}_{i\geq 0}$ of $f$ subordinate to $\{\Omega_i\}_{i\geq 0}$ satisfying \eqref{Decomp estim}. Now, let us assume, to be shown later in this proof, that $\Omega_i$ is included in a ball with radius $R_i=2^{-i+1}$ and  star-shaped with respect to a concentric ball $A_i$ with radius $r_i=2^{-\gamma(i+2)-1}/\gamma$. Then, from Theorem \ref{Costabel Dauge Thm}, there exists a solution of $\div\v^i=f_i$ in $\Omega_i$ that satisfies 
\begin{equation*}
\int_{\Omega_i} |D\v^i(x)|^2\dx\leq \gamma^2 2^{6+4\gamma} 2^{2(\gamma-1)i} \int_{\Omega_i} |f_i(x)|^2\dx.
\end{equation*}

Hence, by extending $\v^i$ by zero, the vector field $\u(x):=\sum_{i\geq 0} \v^i(x)$ satisfies that 
\[\div\u(x)=\div\sum_{i\geq 0} \v^i(x)=\sum_{i\geq 0} f_i(x)=f(x),\]
and 
\begin{align*}
&\int_\Omega |D\u(x)|^2 x_1^{2(\gamma-1)}\omega^{-2}(x_1) \dx\\
&\leq 2\sum_{i\geq 0} \int_{\Omega_i} |D\v^i(x)|^2  x_1^{2(\gamma-1)} \omega^{-2}(x_1) \dx\\
&\leq 2 C_\omega^2 \sum_{i\geq 0} 2^{-2i(\gamma-1)} \omega^{-2}(2^{-i}) \int_{\Omega_i} |D\v^i(x)|^2 \dx\\
&\leq  \gamma^2 2^{7+4\gamma} C_\omega^2 \sum_{i\geq 0} \omega^{-2}(2^{-i}) \int_{\Omega_i} |f_i(x)|^2 \dx\\
&\leq  \gamma^2 2^{7+4\gamma} C_\omega^4 \sum_{i\geq 0} \int_{\Omega_i} |f_i(x)|^2 \omega(x_1)^{-2}\dx\\
&\leq  \gamma^2 2^{12+4\gamma} C_\omega^8 C_H^2  \int_\Omega |f(x)|^2\dx.
\end{align*}

Finally, let us show that $\Omega_i$ is included in a ball with radius $R_i=2^{-i+1}$ and star-shaped with respect to a concentric ball $A_i$. 
Notice that $\Omega_i$ is included in the square $[0,2^{-i}]^2$ with diameter $2^{-i+1/2}$. Thus, any ball with center at a point in $\Omega_i$ and radius $R_i=2^{-i+1}$ contains $\Omega_i$. We define $A_i$ as the ball with radius $r_i:=\rho_i/2\gamma$, and center $c_i:= (2^{-i}-r_i,r_i)$, where $\rho_i=2^{-\gamma(i+2)}$, as shown in Picture \ref{Picture Star}.

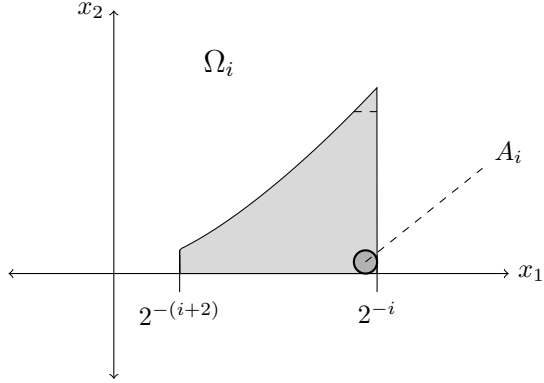
\begin{figure}[h]
\begin{tikzpicture}[xscale=7,yscale=7]
\draw[fill=gray!30!white] 
plot[smooth,samples=100,domain=1/8:1/2] (\x,{(\x)^1.5})  -- 
plot[smooth,samples=100,domain=1/2:1/8] (\x,{0}) --
plot[smooth,samples=100,domain=8^-1.5:0] ({1/8},\x);
\draw[<->] (-.2,0)--(3/4,0) node[right]{$x_1$};
\draw[<->] (0,-0.2)--(0,1/2) node[left]{$x_2$};
\foreach \x in {1/2}, 
    \draw (\x,1pt)--(\x,-1pt) node[below] {$2^{-i}$};
\foreach \x in {1/8}, 
    \draw (\x,1pt)--(\x,-1pt) node[below] {$2^{-(i+2)}$};    
\node at (.2,.4) {\Large $\Omega_i$};
\draw[thick, fill=gray!60!white] (1/2-.5*8^-1.5,.5*8^-1.5) circle (.5*8^-1.5);
\begin{scope}[dashed]
\draw (1/2-.5*8^-1.5,.5*8^-1.5) -- (.7,.2);
\node at (.75,.23) {$A_i$};
\draw (1/2-8^-1.5,.3077299) -- (1/2,.3077299);
\end{scope}
\end{tikzpicture}
\caption{$\Omega_i$ is a star-shaped domain.}\label{Picture Star}
\end{figure}

Now, given $y\in \Omega_i$ and $x\in A_i$, we have to show that the segment $\overline{xy}$ with end-points at $y$ and $x$ is included in $\Omega_i$.   

Now, the open rectangle $D_t$ with sides parallel to the axis and vertices $(2^{-i},0)$ and $(t,t^\gamma)$, for $2^{-i-2}\leq t\leq 2^{-i}-2r_i$, is convex, contains $B_i$ and is included in $\Omega_i$. Thus, the segment $\overline{xy}$ is included in $\Omega_i$ if $y$ belongs to $D_t$, for any $t$ in the interval $[2^{-i-2},2^{-i}-2r_i]$.

Hence, it is sufficient to prove the case where $y=(y_1,y_2)$ belongs to the region above (or over) the dashed line in Picture \ref{Picture Star}: $2^{-i}-2r_i<y_1<2^{-i}$ and $(2^{-i}-2r_i)^\gamma\leq y_2<y_1^\gamma$. Moreover, observe that if the segment $\overline{xy}$ is not included in $\Omega_i$ then its slope must be equal to $\gamma t^{\gamma-1}$, for some $2^{-i}-2r_i<t<2^{-i}$. Hence, it is sufficient to show that the slope of $\overline{xy}$ is larger than $\gamma$ i.e. 
\begin{equation*}
\frac{|y_2-x_2|}{|y_1-x_1|}\geq \gamma.
\end{equation*}
Now, it follows from some straightforward estimations that 
\begin{equation*}
|y_2-x_2|\geq 2^{-i\gamma}-|2^{-i\gamma}-y_2|-|x_2|,
\end{equation*}
where, $|x_2|\leq \rho_i$ and
\begin{equation*}
|2^{-i\gamma}-y_2|\leq |(2^{-i})^\gamma-(2^{-i}-2r_i)^\gamma|<\gamma 2r_i=\rho_i.
\end{equation*}
Then, 
\begin{equation*}
\frac{|y_2-x_2|}{|y_1-x_1|}\geq \frac{2^{2\gamma}\rho_i-2\rho_i}{\rho_i/\gamma}\geq 2\gamma.
\end{equation*}

\end{proof}

\section{The weighted discrete Hardy inequality}
\label{Selena Van Intro}

In this chapter, we prove the two corollaries stated in the Introduction about the solvability of the divergence equation in weighted Sobolev spaces for the weights $\omega(x)=x_1^\beta$ and $\omega(x)=(1-\ln(x_1))^\alpha$. Notice that Theorem \ref{Main} requires $p=2$, however, we analyze the general case $1<p<\infty$ since Theorem \ref{Decomp Thm}, which does not have the constraint $p=2$, can be used to obtain other inequalities (such as the weighted fractional Poincar\'e inequality \cite{HL}) in our cuspidal domain $\Omega$.

Let us recall the characterization of the weighted discrete Hardy inequality proved by K. F.  Andersen and H. P. Heinig. We also refer to  \cite[page 56]{KPS} and \cite{O} for more details.

\begin{theorem}\label{Characterization} Let $\{u_i\}_{i\geq 1}$ and $\{v_i\}_{i\geq 1}$ be sequence weights, and  the conjugate exponents ${1<p,q<\infty}$, i.e. $\frac{1}{p}+\frac{1}{q}=1$, then inequality \eqref{wHardy} is valid if and only if 
\begin{equation*}
A=\sup_{k\geq 1} \left(\sum_{i=k}^\infty u_i\right)^{\frac{1}{p}} \left(\sum_{i=1}^k v_i^{1-q} \right)^{\frac{1}{q}} <\infty.
\end{equation*}
In addition, if $C_H$ represents the optimal constant in \eqref{wHardy}, then
\begin{align*}
A &\leq C_H\leq 4 A
\end{align*}

\end{theorem}

Thus, we use the characterization for the validity of the weighted discrete Hardy inequality to determine the exponents $\beta$ and $\alpha$ for which the previos weights satisfy the sufficient condition in Theorem \ref{Main}: 
\begin{equation*}
\sum_{j=1}^\infty u_i \left(\sum_{i=1}^j d_i \right)^2\leq C_H^2\sum_{j=1}^\infty u_i d_j^2,
\end{equation*}
where 
\[u_i:=|\Omega_i|\,\omega^2(2^{-i}).\] 

Let us start by calculating the measure of the subdomains $\Omega_i$:

\begin{align*}
|\Omega_{i}| &= \int^{2^{-i}}_{2^{-\left(i+2\right)}}\int^{x_1^{\gamma}}_{0}\dx_2\dx_1 = \int^{2^{-i}}_{2^{-\left(i + 2\right)}}x_1^{\gamma}\dx_1\\ 
                     &= \Eval{\frac{x_1^{\gamma + 1}}{\gamma + 1}}{2^{-\left(i + 2\right)}}{2^{-i}}\\ 
                     &= \frac{1}{\gamma + 1} \left({2^{-i(\gamma + 1)}}-2^{-(i+2)(\gamma + 1)}\right)\\
                     &= \frac{1 - 2^{-2(\gamma + 1)}}{\gamma + 1} 2^{-i(\gamma + 1)}\\
                     &=C_\gamma 2^{-(\gamma + 1)i},
\end{align*}
where 
\begin{equation}\label{notation}
C_\gamma = \dfrac{1 - 2^{-2(\gamma + 1)}}{\gamma + 1}. 
\end{equation}

For simplicity, we include some basic calculations on geometric sums which will be used in the following proofs:
\begin{align*}
\sum^{\infty}_{i = k}r^{i}  &= \dfrac{r^{k}}{1 - r}, \text{ for }0<r<1,\\ 
\sum^{k}_{i = 1}r^{i(1 - q)}&= \dfrac{(r^{1-q})^{k + 1} - r^{1-q}}{r^{1-q} - 1}, \text{ for }r>0,\text{ and }q>1.
\end{align*}

The following lemma considers the power weights $\omega(x_1)=x_1^\beta$. 
\begin{lemma}\label{Selena Van 1} Let $\O\subset\R^2$ be the domain defined in \eqref{domain} and $1<p,q<\infty$, with $\frac{1}{p}+\frac{1}{q}=1$. Then, the weight $\omega:\O\to\R$ defined by $\omega(x)=x_1^\beta$, with $\beta>\frac{-\gamma-1}{p}$, is an admissible weight in the sense of Definition \ref{admissible weights}, and satisfies weighted discrete Hardy inequality \eqref{dHardy 2} where $\omega_i:=\omega(2^{-i})$. 

Moreover, 
\begin{equation*}
C_H<4\left(\frac{1}{r(1-r)}\right)^{1/p}\left(\frac{1}{r^{1-q}-1}\right)^{1/q},
\end{equation*}
where 
\begin{equation*}
r:=2^{-p\beta-\gamma-1}.
\end{equation*}
\end{lemma}

\begin{proof}
First, let us show that $x_1^{p\beta}\in L^1(\Omega)$:

\begin{equation*}
\int^{1}_{0}\int^{x_1^{\gamma}}_{0}x_1^{\beta p}\dx_2\dx_1= \int^{1}_{0}x_1^{\beta p + \gamma}dx_1,
\end{equation*}
which is finite if and only if $\beta p + \gamma > -1$, equivalently, $\beta>\frac{-\gamma-1}{p}$. Moreover, it is easy to prove that Condition \eqref{admissible} is valid with $C_\omega=2^{2|\beta_0|}$, where $\beta_0:=\frac{-\gamma-1}{p}$.

Now, we have to show that the weighted discrete Hardy inequality \eqref{dHardy 2} is satisfied for the sequence weight $\omega_i:=2^{-i\beta}$, with $\beta>\frac{-\gamma-1}{p}$. Thus, by Theorem \ref{Characterization}, it is necessary and sufficient to show that 
\begin{equation*}
A=\sup_{k\geq 1} \left(\sum_{i=k}^\infty |\Omega_i|2^{-i\beta p} \right)^{1/p} \left(\sum_{i=1}^k (|\Omega_i|2^{-i\beta p})^{1-q} \right)^{1/q} <\infty.
\end{equation*}

Hence, let us denote 
\begin{equation*}
|\Omega_i| 2^{-i\beta p}= C_\gamma \left(2^{-(\gamma+1)-p\beta}\right)^i=:C_\gamma r^i,
\end{equation*}
where $C_\gamma$ was introduced in \eqref{notation}. Notice that $r\in (0,1)$. Thus,

\begin{align*} A=&\sup_{k \geq 1}\left (\sum^{\infty}_{i = k}C_\gamma r^i\right)^{1/p}\left(\sum^{k}_{i = 1}(C_\gamma r^i)^{(1 - q)}\right)^{1/q}\\
=&C_\gamma^{1/p+(1-q)/q}\sup_{k \geq 1}\left (\sum^{\infty}_{i = k} r^i\right)^{1/p}\left(\sum^{k}_{i = 1}r^{i(1 - q)}\right)^{1/q}\\
=&\sup_{k \geq 1}\left(\dfrac{r^k}{1 - r}\right)^{1/p}\left(\dfrac{(r^{1-q})^{k + 1} - r^{1-q}}{r^{1-q} - 1}\right)^{1/q}\\
=&\left(\frac{1}{1-r}\right)^{1/p}\left(\frac{r^{1-q}}{r^{1-q}-1}\right)^{1/q}\sup_{k \geq 1} r^{k/p} \left(r^{k(1-q))}-1\right)^{1/q} \\
<&\left(\frac{1}{1-r}\right)^{1/p}\left(\frac{r^{1-q}}{r^{1-q}-1}\right)^{1/q}\sup_{k \geq 1} r^{k/p} r^{k(1-q)/q}\\
=&\left(\frac{1}{1-r}\right)^{1/p}\left(\frac{r^{1-q}}{r^{1-q}-1}\right)^{1/q}<\infty.
\end{align*}
Moreover, using again Theorem \ref{Characterization}, it follows that 
\begin{equation*}
C_H\leq 4A<4\left(\frac{1}{r(1-r)}\right)^{1/p}\left(\frac{1}{r^{1-q}-1}\right)^{1/q},
\end{equation*}
where 
\begin{equation*}
r:=2^{-p\beta-\gamma-1}.
\end{equation*}
\end{proof}

\begin{proof}[Proof of Corollary \ref{Corollary 1}] It follows from Theorem \ref{Main} and Lemma \ref{Selena Van 1}.
\end{proof}

\begin{lemma}\label{Selena Van 3} Let $\O\subset\R^2$ be the domain defined in \eqref{domain} and $1<p,q<\infty$, with $\frac{1}{p}+\frac{1}{q}=1$. Then, the weight $\omega:\O\to\R$ defined by $\omega(x)=(1-\ln(x_1))^\alpha$, with $\alpha\in\R$, is an admissible weight in the sense of Definition \ref{admissible weights}, and satisfies the weighted discrete Hardy inequality \eqref{dHardy 2} for $\omega_i:=\omega(2^{-i})$. 
\end{lemma}

\begin{proof} If $\alpha$ is zero, then $\omega(x)=1$. This weight was studied  in Lemma \ref{Selena Van 1}, for $\beta=0$, which is admissible and satisfies the discrete Hardy inequality \eqref{dHardy 2} with 
\begin{equation*}
C_H<4\left(\frac{1}{r(1-r)}\right)^{1/p}\left(\frac{1}{r^{1-q}-1}\right)^{1/q},
\end{equation*}
for 
\begin{equation*}
r:=2^{-(\gamma+1)}.
\end{equation*}

Thus, we have to consider the case when $\alpha$ is different from $0$.

First, let us show that  ${\omega^p(x)}=(1-\ln(x_1))^{p\alpha}\in L^1(\Omega)$:
\begin{align*}
&\int^{1}_{0}\int^{x_1^{\gamma}}_{0} (1-\ln(x_1))^{p\alpha}\dx_2\dx_1\\
&= \int^{1}_{0} {x_1^{\gamma}}(1 - \ln(x_1))^{p\alpha} \dx_1.
\end{align*}

If $\alpha$ is positive, then the function $f(x_1)=x_1^{\gamma}((1 - \ln(x_1))^{p\alpha}$ tends to 0 as $x_1$ tends to 0 from the right, then the integral of this continuous function is finite:
\begin{align*}
\lim_{{x_1} \to 0^+} x_1^{\gamma}(1 - \ln(x_1))^{p\alpha}&=\lim_{{x_1} \to 0^+} \left(\frac{1-\ln(x_1)}{x_1^{-\gamma/{p\alpha}}}\right)^{p\alpha}=0,
\end{align*}
since
\begin{align*}\lim_{{x_1} \to 0^+} \frac{1-\ln(x_1)}{x_1^{-\gamma/{p\alpha}}}&=\lim_{{x_1} \to 0^+} \frac{-x_1^{-1}}{(-\gamma/p\alpha)x_1^{-\gamma/{p\alpha-1}}}
&=\lim_{{x_1} \to 0^+} \frac{p\alpha}{\gamma} x_1^{\gamma/p\alpha}=0.
\end{align*}

If $\alpha$ is negative then $0<x_1^{\gamma}(1 - \ln(x_1))^{p\alpha}<1$, thus ${\omega^p(x)}\in L^1(\Omega)$. 

Now, let us estimate the constant $C_\omega$ in inequality \eqref{admissible}: 
\begin{equation*}
\sup_{x\in \Omega_i} \omega(x)\leq C_\omega \inf_{x\in \Omega_i} \omega(x).
\end{equation*}

If $\alpha$ is positive, then $\omega(x)$ is decreasing with respect to $x_1$, then 
\begin{align*}
\sup_{x\in \Omega_i}\omega(x)&= \text{ } \omega(2^{-i-2})=(1+(i+2)\ln(2))^\alpha\\
\inf_{x\in \Omega_i}\omega(x)&= \text{ } \omega(2^{-i})=(1+i\ln(2))^\alpha,
\end{align*}
hence, 
\begin{align*}
\frac{\omega(2^{-i-2})}{\omega(2^{-i})}&=\left(1+\frac{2\ln(2)}{1+i\ln(2)}\right)^\alpha\leq \left(1+ 2\ln(2)\right)^\alpha.
\end{align*}

If $\alpha$ is negative, then $\omega(x)$ is increasing with respect to $x_1$, then 
\begin{align*}
\sup_{x\in \Omega_i}\omega(x)&= \text{ } \omega(2^{-i})=(1+i\ln(2))^\alpha\\
\inf_{x\in \Omega_i}\omega(x)&= \text{ } \omega(2^{-i-2})=(1+(i+2)\ln(2))^\alpha,
\end{align*}
hence, 
\begin{align*}
\frac{\omega(2^{-i})}{\omega(2^{-i-2})}&=\left(1+\frac{2\ln(2)}{1+i\ln(2)}\right)^{-\alpha} \leq \left(1+ 2\ln(2)\right)^{-\alpha}.
\end{align*}
Thus,  $C_\omega:= \left(1+ 2\ln(2)\right)^{|\alpha|}$ satisfies estimate \eqref{admissible}.

Third, let us study the weighted discrete Hardy inequality for this weight. We use the characterization stated in Theorem \ref{Characterization} thus we have to estimate the following supremum
\begin{multline}\label{A for Log}
A= \sup_{k\geq 1}\left(\sum_{i=k}^\infty {u_i}\right)^\frac{1}{p}\left(\sum_{i=1}^k {u_i}^{1-q}\right)^\frac{1}{q}\\
=\sup_{k\geq 1}\left(\sum_{i=k}^\infty 2^{-{(\gamma+1)}i}\left(1+i\ln(2)\right)^{p\alpha}\right)^\frac{1}{p}\\
\left(\sum_{i=1}^k \left(2^{-{(\gamma+1)}i}\left(1+i\ln(2)\right)^{p\alpha}\right)^{1-q}\right)^\frac{1}{q}.
\end{multline}

If $\alpha$ is negative, then $p\alpha<0$ and $(1+i\ln(2))^{p\alpha}\leq (1+k\ln(2))^{p\alpha}$ for all $i\geq k$. Similarly, $p\alpha(1-q)>0$ and $(1+i\ln(2))^{p\alpha(1-q)}\leq (1+k\ln(2))^{p\alpha(1-q)}$ for all $i\leq k$. Thus, 
\begin{align*}
A\leq &\sup_{k\geq 1} (1+k\ln(2))^{\alpha+p\alpha(1-q)/q} \left(\sum_{i=k}^\infty r^i \right)^\frac{1}{p}\left(\sum_{i=1}^k r^{i(1-q)}\right)^\frac{1}{q}\\
=&\sup_{k\geq 1} \left(\sum_{i=k}^\infty r^i \right)^\frac{1}{p}\left(\sum_{i=1}^k r^{i(1-q)}\right)^\frac{1}{q}\\ 
\leq& \left(\frac{1}{1-r}\right)^{1/p}\left(\frac{r^{1-q}}{r^{1-q}-1}\right)^{1/q}<\infty,
\end{align*}
where $r=2^{-(\gamma+1)}$.

For $\alpha$ positive, 
we define $a=p\alpha>0$ and $f(t)=r^t (1+t\ln(2))^a$. By a straightforward calculation, it can be seen that $f$ is positive and  decreasing for $t$ sufficiently large. Thus, there exists $k_0\in\N$ such 
\begin{align*}
\sum_{i=k}^\infty 2^{-{(\gamma+1)}i}\left(1+i\ln(2)\right)^{p\alpha}\leq \int_{k-1}^\infty r^t (1+t\ln(2))^a \d t.
\end{align*}
for $k\geq k_0$. Next, by using integration by parts, we obtain
\begin{align*}
I&:= \int_{k-1}^\infty r^t (1+t\ln(2))^a \d t\\
&\leq \frac{-r^k}{r\ln(r)}(1+k\ln(2))^a+\int_{k-1}^\infty r^t (1+t\ln(2))^a \left[\frac{a\ln(2)}{-\ln(r)(1+t\ln(2))}\right]\d t.
\end{align*}
Now, we assume that $k_0$ is sufficiently large such that the function between brackets in the previous line is less than $1/2$. Thus, 
\begin{align*}
I\leq \frac{-r^k}{r\ln(r)}(1+k\ln(2))^a+\frac{1}{2}I,
\end{align*}
and
\begin{align*}
\frac{1}{2} I\leq \frac{-r^k}{r\ln(r)}(1+k\ln(2))^a.
\end{align*}
Thus, it follows that 
\begin{align}
\sum_{i=k}^\infty r^i (1+i\ln(2))^a &\leq \int_{k-1}^\infty r^t (1+t\ln(2))^a \d t=I\notag\\ 
\label{Estimate 1}& \leq \frac{-2r^k}{r\ln(r)} (1+k\ln(2))^a,
\end{align}
for $k\geq k_0$. 

Let us study the second sum in the estimation of $A$ in \eqref{A for Log}. Thus, we define 
\begin{equation}\label{tilde r}
{\tilde{r}}:= 2^{-(\gamma+1)(1-q)}= r^{1-q}= r^{\frac{-q}{p}} >1,
\end{equation} 
and 
\begin{equation}\label{tilde a}
{\tilde{a}}:= a(1-q)= {-aq}/{p} <0.
\end{equation}
Notice that the function  $g(t):={\tilde{r}}^t (1+t\ln(2))^{\tilde{a}}$ is positive and increasing for $t$ sufficiently large. Thus, there exists a constant $C_2>1$ such that 
\begin{align}\label{Estimate 2}
\sum_{i=1}^k \tilde{r}^i\left(1+i\ln(2)\right)^{\tilde{a}} \leq C_2 \int_1^{k+1} \tilde{r}^t (1+t\ln(2))^{\tilde{a}} \d t,
\end{align}
for all $k\geq 1$.
Finally, notice that to show that $A$ in \eqref{A for Log} is finite  it is sufficient to consider the case where the supremum runs over $k\geq k_0$ and estimate its power $q$. Thus, from \eqref{Estimate 1} and \eqref{Estimate 2}, we have 
\begin{multline*}
\sup_{k \geq k_0} \left(\sum_{i=k}^\infty r^i\left(1+i\ln(2)\right)^{a}\right)^{\frac{q}{p}} \left(\sum_{i=1}^k \tilde{r}^i\left(1+i\ln(2)\right)^{\tilde{a}}\right)\\
\leq C_2 \frac{\int_1^{k+1} \tilde{r}^t (1+t\ln(2))^{\tilde{a}} dt}{r^{\frac{-kq}{p}} (1+k\ln(2))^{\frac{-aq}{p}}},
\end{multline*}
for another constant $C_2$, which is independent of $k$, denoted with the same letter for simplicity.

Finally, we calculate the limit of the above quotient as $k$ goes to infinity, understanding $k$ as a continuous variable. We use for this analysis definitions \eqref{tilde r} and \eqref{tilde a}, and L'Hospital rule:
\begin{align*}
&\lim_{k \to \infty} \frac{\int_1^{k+1} \tilde{r}^t (1+t\ln(2))^{\tilde{a}} dt}{r^{\frac{-kq}{p}} (1+k\ln(2))^{\frac{-aq}{p}}}=\lim_{k \to \infty} \frac{\int_1^{k+1} \tilde{r}^t (1+t\ln(2))^{\tilde{a}} dt}{\tilde{r}^{k} (1+k\ln(2))^{\tilde{a}}}\\
=&\lim_{k \to \infty} \frac{\tilde{r}^{k+1} (1+(k+1)\ln(2))^{\tilde{a}}}{\ln(\tilde{r})\tilde{r}^{k} (1+k\ln(2))^{\tilde{a}} +\tilde{a} \ln(2) \tilde{r}^{k} (1+k\ln(2))^{\tilde{a}-1}}\\
=&\lim_{k \to \infty} \tilde{r} \left(\frac{1+(k+1)\ln(2)}{1+k\ln(2)}\right)^{\tilde{a}} \frac{1}{\ln(\tilde{r})+\frac{\tilde{a}\ln(2)}{1+k\ln(2)}}\\
=&\frac{\tilde{r}}{\ln(\tilde{r})}.
\end{align*} 

Therefore, the sequence is convergent and bounded, which implies that $A$ is finite. 
\end{proof}

\begin{proof}[Proof of Corollary \ref{Corollary 2}] It follows immediately from Theorem \ref{Main} and Lemma \ref{Selena Van 3}.
\end{proof}

\section*{Acknowledgements}
We gratefully acknowledge Professor M. Helena Noronha from Cal State Northridge for creating the PUMP program {\bf DMS-1247679} depending on NSF, which provided support to several students in the California State Universities located in Southern California to conduct research in Mathematics. We extend our acknowledgement to Professor John Rock from Cal Poly Pomona for introducing the authors to Professor M. Helena Noronha and this outstanding program.


\begin{thebibliography}{Lg}


\bibitem{AD}{\sc G. Acosta and R. Dur\'an},
\emph{Divergence operator and related inequalities},
Springer Briefs in Mathematics, Springer, New York, 2017.

\bibitem{ADL}{\sc G. Acosta, R. Dur\'an and F. L\'opez-Garc\'{\i}a},
\emph{Korn inequality and divergence operator: counterexamples and optimality of weighted estimates},  
Proc. of the AMS. {\bf 141} (2013), 217--232.

\bibitem{AH}{\sc K. F. Andersen and H. P. Heinig},
\emph{Weighted norm inequalities for certain integral operators},
SIAM J. Math. Anal. {\bf 14}, 4 (1983), 834--844.

\bibitem{AO}{\sc G. Acosta and I. Ojea},
\emph{Korn's inequalities for generalized external cusps},
Math. Methods Appl. Sci. {\bf 39} (2016), 4935--4950.

\bibitem{B}{\sc G. Bennett}, 
\emph{Some Elementary Inequalities III.}, 
Quart. J. Math. Oxford Ser. (2), {\bf 42} (1991),149--174.

\bibitem{CD}{\sc M. Costabel and M. Dauge},
\emph{On the inequalities of {B}abu\v ska-{A}ziz, {F}riedrichs and {H}organ-{P}ayne},
Arch. Ration. Mech. Anal. {\bf 217}, 3 (2015), 873--898.

\bibitem{DL}{\sc R. Dur\'an and F. L\'opez-Garc\'\i a},
\emph{Solutions of the divergence and {K}orn inequalities on domains with an external cusp},
Ann. Acad. Sci. Fenn. Math., {\bf 35} (2010), 421--438.

\bibitem{HP}{\sc C. O. Horgan and L. E. Payne},
\emph{On inequalities of {K}orn, {F}riedrichs and {B}abu\v{s}ka-{A}ziz},
Arch. Rational Mech. Anal. {\bf 82}, 2 (1983), 165--179.

\bibitem{HL}{\sc R. Hurri-Syrj\"anen and F. L\'opez-Garc\'\i a},
\emph{On the weighted fractional Poincare-type inequalities},
to appear in Colloquium Mathematicum (2019). DOI: 10.4064/cm7612-9-2018.

\bibitem{JK}{\sc R. Jiang and A. Kauranen},
\emph{Korn's inequality and {J}ohn domains},
Calc. Var. Partial Differential Equations {\bf 56}, 4 (2017).

\bibitem{KPS}{\sc A. Kufner, L. E. Persson and N. Samko},
\emph{Weighted inequalities of Hardy type},
Second edition, World Scientific Publishing Co. Pte. Ltd., Hackensack, NJ, 2017.

\bibitem{L1} {\sc F. L\'opez-Garc\'ia}, 
\emph{A decomposition technique for integrable functions with applications to the divergence problem}, 
J. Math. Anal. Appl. {\bf 418} (2014), 79--99.

\bibitem{L2}{\sc F. L\'opez-Garc\'\i a},
\emph{Weighted {K}orn inequalities on {J}ohn domains},
Studia Math. {\bf 241},1 (2018), 17--39.

\bibitem{L3}{\sc F. L\'opez-Garc\'\i a},
\emph{Weighted Generalized {K}orn inequalities on {J}ohn domains},
Math. Methods Appl. Sci. {\bf 41}, 17 (2018), 8003--8018.

\bibitem{O}{\sc C. A. Okpoti},
\emph{Weight Characterizations of Discrete Hardy and Carleman Type Inequalities}, 
Licentiate Thesis, Lule\aa\ University of Technology, Department of Mathematics, Sweden, 2005.
\end{thebibliography}
\end{document}